\documentclass[12pt]{article}
\usepackage[margin=1in]{geometry}  % set the margins to 1in on all sides
\usepackage{graphicx}              % to include figures
\usepackage{amsmath}               % great math stuff
\usepackage{amsthm}                % better theorem environments
\usepackage[ansinew]{inputenc}
\usepackage{array,color}
\usepackage{amsxtra}
\usepackage{amstext}
\usepackage{latexsym}
\usepackage{dsfont}
\usepackage[all]{xy}
\usepackage{amsmath,amsfonts,amssymb}
\usepackage{cite}

\usepackage{authblk}

\usepackage{enumitem}  % http://ctan.org/pkg/enumitem
\usepackage{calc}% http://ctan.org/pkg/calc
\setlist{labelindent=1pt,itemsep=.5em}
\setlist[itemize]{leftmargin=1.2cm}
\setlist[enumerate]{itemindent=0em,leftmargin=1.2cm}
\setlist[enumerate,1]{label={\upshape(\roman*)}}

\allowdisplaybreaks

\makeatletter
\newcommand{\subjclass}[2][2020]{%
  \let\@oldtitle\@title%
  \gdef\@title{\@oldtitle\footnotetext{#1 \emph{Mathematics subject classification}: #2}}%
}
\newcommand{\keywords}[1]{%
  \let\@@oldtitle\@title%
  \gdef\@title{\@@oldtitle\footnotetext{\emph{Keywords}: #1}}%
}
\makeatother

\newtheorem{thm}{Theorem}[section]
\newtheorem{cor}[thm]{Corollary}

\newtheorem{prop}[thm]{Proposition}
\theoremstyle{definition}
\newtheorem{defn}[thm]{Definition}
\theoremstyle{remark}

\newtheorem{exm}[thm]{Example}
\numberwithin{equation}{section}

\DeclareMathOperator{\id}{id}
\title{Color Hom-Lie algebras, color Hom-Leibniz algebras and color omni-Hom-Lie algebras}
\author{Abdoreza Armakan$^{1}$, Sergei Silvestrov$^{2}$ \\
\small{$^{1}$Department of Mathematics, College of Sciences, Shiraz University, \authorcr
P.O. Box 71457-44776, Shiraz, Iran. \authorcr
e-mail: r.armakan@shirazu.ac.ir, reza.armakan@gmail.com \authorcr
$^{2}$ Division of Applied Mathematics,
School of Education, Culture and Communication, \authorcr
M\"{a}lardalen University, Box 883, 72123 V\"{a}steras, Sweden. \authorcr
e-mail: sergei.silvestrov@mdh.se}}


\date{}

\begin{document}

\maketitle

%\footnote[0]{{\it Corresponding authors}: Abdoreza Armakan, Sergei Silvestrov}

\abstract{Representations of color Hom-Lie algebras are reviewed, and it is shown that there exist a series of coboundary operators. We also introduce the notion of a color omni-Hom-Lie algebra associated to a vector space and an even invertible linear map. We show how regular color Hom-Lie algebra structures on a vector space can be characterized. Moreover, it is shown that the underlying algebraic structure of the color omni-Hom-Lie algebra is a color Hom-Leibniz algebra.}

\section{Introduction}

The investigations of various quantum deformations or $q$-deformations of Lie algebras began a period of rapid expansion in 1980's stimulated by introduction of quantum groups motivated by applications to the quantum Yang-Baxter equation, quantum inverse scattering methods and constructions of the quantum deformations of universal enveloping algebras of semi-simple Lie algebras. Various $q$-deformed Lie algebras have appeared in physical contexts such as string theory, vertex models in conformal field theory, quantum mechanics and quantum field theory in the context of deformations of infinite-dimensional algebras, primarily the Heisenberg algebras, oscillator algebras and Witt and Virasoro algebras. In \cite{AizawaSaito,ChaiElinPop,ChaiIsLukPopPresn,ChaiKuLuk,ChaiPopPres,CurtrZachos1,DamKu,DaskaloyannisGendefVir,Hu,Kassel92,LiuKQuantumCentExt,LiuKQCharQuantWittAlg,LiuKQPhDthesis},
it was in particular discovered that in these $q$-deformations of Witt and Visaroro algebras and some related algebras, some interesting $q$-deformations of Jacobi identities, extending Jacobi identity for Lie algebras, are satisfied. This has been one of the initial motivations for the development of general quasi-deformations and discretizations of Lie algebras of vector fields using more general $\sigma$-derivations (twisted derivations) in \cite{HLS}.

Hom-Lie algebras and more general quasi-Hom-Lie algebras were introduced first by Larsson, Hartwig and Silvestrov \cite{HLS}, where the general quasi-deformations and discretizations of Lie algebras of vector fields using more general $\sigma$-derivations (twisted derivations) and a general method for construction of deformations of Witt and Virasoro type algebras based on twisted derivations have been developed, initially motivated by the $q$-deformed Jacobi identities observed for the $q$-deformed algebras in physics, along with $q$-deformed versions of homological algebra and discrete modifications of differential calculi. Hom-Lie algebras, Hom-Lie superalgebras, Hom-Lie color algebras and more general quasi-Lie algebras and color quasi-Lie algebras where introduced first in \cite{LarssonSilv2005:QuasiLieAlg,LarssonSilv:GradedquasiLiealg,SigSilv:CzechJP2006:GradedquasiLiealgWitt}. Quasi-Lie algebras and color quasi-Lie algebras encompass within the same algebraic framework the quasi-deformations and discretizations of Lie algebras of vector fields by $\sigma$-derivations obeying twisted Leibniz rule, and the well-known generalizations of Lie algebras such as color Lie algebras, the natural generalizations of Lie algebras and Lie superalgebras. In quasi-Lie algebras, the skew-symmetry and the Jacobi identity are twisted by deforming twisting linear maps, with the Jacobi identity in quasi-Lie and quasi-Hom-Lie algebras in general containing six twisted triple bracket terms. In Hom-Lie algebras, the bilinear product satisfies the non-twisted skew-symmetry property as in Lie algebras, and the Hom-Lie algebras Jacobi identity has three terms twisted by a single linear map, reducing to the Lie algebras Jacobi identity when the twisting linear map is the identity map. Hom-Lie admissible algebras have been considered first in \cite{ms:homstructure}, where in particular the Hom-associative algebras have been introduced and shown to be Hom-Lie admissible, that is leading to Hom-Lie algebras using commutator map as new product, and in this sense constituting a natural generalization of associative algebras as Lie admissible algebras. Since the pioneering works \cite{HLS,LarssonSilvJA2005:QuasiHomLieCentExt2cocyid,LarssonSilv:GradedquasiLiealg,LarssonSilv2005:QuasiLieAlg,LarssonSilv:QuasidefSl2,ms:homstructure}, Hom-algebra structures expanded into a popular area with increasing number of publications in various directions. Hom-algebra structures of a given type include their classical counterparts and open broad possibilities for deformations, Hom-algebra extensions of cohomological structures and representations, formal deformations of Hom-associative and Hom-Lie algebras, Hom-Lie admissible Hom-coalgebras, Hom-coalgebras, Hom-Hopf algebras \cite{AmmarEjbehiMakhlouf:homdeformation,BenMakh:Hombiliform,ElchingerLundMakhSilv:BracktausigmaderivWittVir,LarssonSilvJA2005:QuasiHomLieCentExt2cocyid,LarssonSilvestrovGLTMPBSpr2009:GenNComplTwistDer,MakhSil:HomHopf,MakhSilv:HomAlgHomCoalg,MakhSilv:HomDeform,Sheng:homrep,Yau:HomolHom,Yau:EnvLieAlg}.
Hom-Lie algebras, Hom-Lie superalgebras and color Hom-Lie algebras and their $n$-ary generalizations have been further investigated in various aspects for example in \cite{AbramovSilvestrov:3homLiealgsigmaderivINvol,AmmarEjbehiMakhlouf:homdeformation,AmmarMakhloufHomLieSupAlg2010,AmmarMakhloufSaadaoui2013:CohlgHomLiesupqdefWittSup,AmmarMakhloufSilv:TernaryqVirasoroHomNambuLie,ArmakanFarhangdoost:IJGMMP,ArmakanSilv:envelalgcertaintypescolorHomLie,ArmakanSilvFarh:envelopalgcolhomLiealg,ArmakanSilvFarh:exthomLiecoloralg,
akms:ternary,ams:ternary,ArnlindMakhloufSilvnaryHomLieNambuJMP2011,Bakayoko2014:ModulescolorHomPoisson,BakayokoDialo2015:genHomalgebrastr,BakyokoSilvestrov:Homleftsymmetriccolordialgebras,BakyokoSilvestrov:MultiplicnHomLiecoloralg,BakayokoToure2019:genHomalgebrastr,
BenMakh:Hombiliform,CaoChen2012:SplitregularhomLiecoloralg,GuanChenSun:HomLieSuperalgebras,KitouniMakhloufSilvestrov,kms:solvnilpnhomlie2020,LarssonSigSilvJGLTA2008:QuasiLiedefFttN,MabroukNcibSilvestrov2020:GenDerRotaBaxterOpsnaryHomNambuSuperalgs,
ms:homstructure,MakhSilv:HomDeform,MakhSil:HomHopf,MakhSilv:HomAlgHomCoalg,Makhlouf2010:ParadigmnonassHomalgHomsuper,MandalMishra:HomGerstenhaberHomLiealgebroids,MishraSilvestrov:SpringerAAS2020HomGerstenhalgsHomLiealgds,RichardSilvestrovJA2008,RichardSilvestrovGLTbdSpringer2009,
ShengBai2014:homLiebialg,SigSilv:CzechJP2006:GradedquasiLiealgWitt,ShengChen2013:HomLie2algebras,Sheng:homrep,SigSilv:GLTbdSpringer2009,SilvestrovParadigmQLieQhomLie2007,Yau2009:HomYangBaxterHomLiequasitring,Yau:EnvLieAlg,Yau:HomolHom,Yau:HomBial,Yuan2012:HomLiecoloralgstr,ZhouChenMa:GenDerHomLiesuper}.

In Section \ref{sec:defin:homasshommodcolorhomLie} we review basic concepts of Hom-associative algebras, Hom-modules and color Hom-Lie algebras.
In Section \ref{sec:stdifgrcomcolalg}, $(\sigma,\tau)$-differential graded commutative color algebras are defined and the classical result about the relation between Lie algebra structures and differential graded commutative color algebras structures is generalized to relation between color Hom-Lie-algebras and $(\sigma,\tau)$-differential graded commutative color algebras. In Section \ref{sec:ReprescolorHomLiealg}, representations of color Hom-Lie algebras are considered, adjoint representation and its morphism interpretations are investigated, and Hom-cochains, coboundary operators and cohomological complex are described, generalizing some results in \cite{AmmarMabroukMakhloufCohomnaryHNLalg2011,Sheng:homrep}. In Section \ref{sec:colorOmniHomLiealg}, the notion of a color omni-Hom-Lie algebra associated to a vector space and an even invertible linear map is introduced, and it is shown that the color Hom-Leibniz algebras appear as underlying algebraic structure of the color omni-Hom-Lie algebras.

\section{Hom-associative algebras, Hom-modules and color Hom-Lie algebras}
\label{sec:defin:homasshommodcolorhomLie}
We start by recalling some basic concepts from \cite{MakhSilv:HomAlgHomCoalg,ms:homstructure} where also various examples and properties of Hom-associative algebras can be found. Throughout this paper, we use $\mathbf{k}$ to denote a commutative unital ring e.g. a field.
\begin{defn}
\begin{enumerate}[label=\upshape{(\roman*)},left=7pt]
\item A Hom-module is a pair $(M,\alpha)$ consisting of an $\mathbf{k}$-module $M$ and a linear operator $\alpha:M\rightarrow M$.
\item A Hom-associative algebra is a triple $(A,\cdot,\alpha)$ consisting of an $\mathbf{k}$-module $A$, a linear map $\cdot:A\otimes A\rightarrow A$ called the multiplication and a multiplicative linear operator $\alpha:A\rightarrow A$ which satisfies the Hom-associativity condition, namely
$$\alpha(x)\cdot(y\cdot z)=(x\cdot y)\cdot\alpha(z),$$
for all $x,y,z\in A$.
\item A Hom-associative algebra or a Hom-module is called involutive if $\alpha^{2}=\id$.
\item Let $(M,\alpha)$ and $(N,\beta)$ be two Hom-modules. A $\mathbf{k}$-linear map $f:M\rightarrow N$ is called a morphism of Hom-modules if
$$f(\alpha(x))=\beta(f(x)),$$
for all $x\in M$.
\item Let $(A,\cdot,\alpha)$ and $(B,\bullet,\beta)$ be two Hom-associative algebras. A $\mathbf{k}$-linear map $f:A\rightarrow B$ is called a morphism of Hom-associative algebras if
\begin{enumerate}[label=\upshape{\arabic*)},left=7pt]
\item $f(x\cdot y)=f(x)\bullet f(y),$
\item $f(\alpha(x))=\beta(f(x)),$
for all $x,y\in A$.
\end{enumerate}
\item If $(A,\cdot,\alpha)$ is a Hom-associative algebra, then $B\subseteq A$ is called a Hom-associative subalgebra of $A$ if it is closed under the multiplication $\cdot$ and $\alpha(B)\subseteq B$. A submodule $I$ is called a Hom-ideal of A if $x\cdot y \in I$ and $x\cdot y \in I$ for all $x\in I$ and $y \in A$, and also $\alpha(I)\subseteq I$.
\end{enumerate}
\end{defn}

A Hom-Lie algebra is called a {\it multiplicative} Hom-Lie algebra if $\alpha$ is an algebraic morphism, i.e. for any $x,y\in \mathfrak{g}$,
$$\alpha([x,y])=[\alpha(x),\alpha(y)].$$
We call a Hom-Lie algebra {\it regular} if $\alpha$ is an automorphism.

A sub-vector space $\mathfrak{h}\subset \mathfrak{g}$ is a Hom-Lie sub-algebra of $(\mathfrak{g},[\cdot,\cdot], \alpha)$ if $\alpha(\mathfrak{h})\subset \mathfrak{h}$ and $\mathfrak{h}$ is closed under the bracket operation, i.e.
$$[x_{1},x_{2}]_{\mathfrak{g}}\in \mathfrak{h},$$
for all $x_{1},x_{2}\in \mathfrak{h}.$
Let $(\mathfrak{g},[\cdot,\cdot], \alpha)$ be a multiplicative Hom-Lie algebra. Let $\alpha^{k}$ denote the $k$-times composition of $\alpha$ by itself, for any nonnegative integer $k$, i.e.
$$\alpha^{k}=\alpha \circ \dots \circ \alpha ~~~~~(k-times),$$
where we define $\alpha^{0}=Id$ and $\alpha^{1}=\alpha$. For a regular Hom-Lie algebra $\mathfrak{g}$, let
$$\alpha^{-k}=\alpha^{-1} \circ \dots \circ \alpha^{-1}~~~~~(k-times).$$

We now recall the notion of a color Hom-Lie algebra step by step in order to indicate them as a generalization of Lie color algebras.
\begin{defn}[\cite{BahturinMikhPetrZaicevIDLSbk92,ChenSilvestrovOystaeyen:RepsCocycleTwistsColorLie,ChenPetitOystaeyenCOHCHLA,MikhZolotykhCALSbk95,PiontkovskiSilvestrovC3dCLA,ScheunertGLA,ScheunertGTC,ScheunertCOH2,ScheunertZHA,Silvestrov:class3dimcolLiealg}]
Given a commutative group $\Gamma$ (referred to as the grading group), a commutation factor on $\Gamma$ with values in the multiplicative group $K\setminus \{0\}$ of a field $K$ of characteristic 0 is a map
$$\varepsilon: \Gamma \times \Gamma \rightarrow K\setminus \{0\},$$
satisfying three properties:
\begin{enumerate}[label=\upshape{(\roman*)},left=7pt]
  \item  $\varepsilon(\alpha+\beta,\gamma)=\varepsilon(\alpha,\gamma)\varepsilon(\beta,\gamma),$
  \item  $\varepsilon(\alpha,\gamma+\beta)=\varepsilon(\alpha,\gamma)\varepsilon(\alpha,\beta),$
  \item  $\varepsilon(\alpha,\beta)\varepsilon(\beta,\alpha)=1.$
\end{enumerate}
A $\Gamma$-graded $\varepsilon$-Lie algebra (or a Lie color algebra)
\cite{BahturinMikhPetrZaicevIDLSbk92,ChenSilvestrovOystaeyen:RepsCocycleTwistsColorLie,ChenPetitOystaeyenCOHCHLA,MikhZolotykhCALSbk95,PiontkovskiSilvestrovC3dCLA,ScheunertGLA,ScheunertGTC,ScheunertCOH2,ScheunertZHA,Silvestrov:class3dimcolLiealg}
is a $\Gamma$-graded linear space
$$X=\bigoplus_{\gamma\in \Gamma} X_{\gamma},$$
with a bilinear multiplication (bracket) $[\cdot,\cdot]:X\times X \rightarrow X$ satisfying the following properties:
\begin{enumerate}[label=\upshape{(\roman*)},left=7pt]
\item  \textbf{Grading axiom:} $[X_{\alpha}, X_{\beta}]\subseteq X_{\alpha+\beta},$
\item  \textbf{Graded skew-symmetry:} $[a,b]=-\varepsilon(\alpha,\beta)[b,a],$
\item   \textbf{Generalized Jacobi identity:}  \\$\varepsilon(\gamma,\alpha)[a,[b,c]]+ \varepsilon(\beta,\gamma)[c,[a,b]]+ \varepsilon(\alpha,\beta)[b,[c,a]]=0,$
\end{enumerate}
for all $a\in X_{\alpha}, b\in X_{\beta}, c\in X_{\gamma}$ and $\alpha,\beta,\gamma \in \Gamma$.
The elements of $X_{\gamma}$ are called homogenous of degree $\gamma$, for all $\gamma\in \Gamma$.
\end{defn}

Color Hom-Lie algebras are a special class of general color quasi-Lie algebras ($\Gamma$-graded quasi-Lie algebras) defined first in \cite{LarssonSilv2005:QuasiLieAlg,LarssonSilv:GradedquasiLiealg,SigSilv:CzechJP2006:GradedquasiLiealgWitt}.

\begin{defn}[\cite{AbdaouiAmmarMakhloufCohhomLiecolalg2015,CaoChen2012:SplitregularhomLiecoloralg,CalderonDelgado2012:splLiecolor,LarssonSilv2005:QuasiLieAlg,LarssonSilv:GradedquasiLiealg,SigSilv:CzechJP2006:GradedquasiLiealgWitt,Yuan2012:HomLiecoloralgstr}] \label{HLCAD}
A color Hom-Lie algebra is a quadruple $(\mathfrak{g},[\cdot,\cdot],\varepsilon,\alpha)$ consisting of a $\Gamma$-graded vector space $\mathfrak{g}$, a bi-character $\varepsilon$, an even bilinear mapping $$[\cdot,\cdot]:\mathfrak{g}\times \mathfrak{g}\rightarrow \mathfrak{g},$$ (i.e. $[\mathfrak{g}_{a},\mathfrak{g}_{b}]\subseteq \mathfrak{g}_{a+b}$, for all $a,b \in \Gamma$) and an even homomorphism $\alpha:\mathfrak{g}\rightarrow \mathfrak{g}$ such that for homogeneous elements $x,y,z\in \mathfrak{g}$ we have
\begin{enumerate}[label=\upshape{\arabic*)},left=7pt]
  \item  \textbf{$\varepsilon$-skew symmetry:} $[x,y]=-\varepsilon(x,y)[y,x],$
  \item  \textbf{$\varepsilon$-Hom-Jacobi identity:} $\displaystyle{\sum_{cyclic\{x,y,z\}}}\varepsilon(z,x)[\alpha(x),[y,z]]=0.$
\end{enumerate}
\end{defn}

Let $\mathfrak{g}=\bigoplus_{\gamma\in \Gamma}\mathfrak{g}_{\gamma}$ and $\mathfrak{h}=\bigoplus_{\gamma\in \Gamma}\mathfrak{h}_{\gamma}$ be two $\Gamma$-graded color Lie algebras. A linear mapping $f:\mathfrak{g}\rightarrow \mathfrak{h}$ is said to be homogenous of the degree $\mu\in \Gamma$ if
$$f(\mathfrak{g}_{\gamma})\subseteq \mathfrak{h}_{\gamma+\mu},$$
for all $\gamma\in \Gamma$. If in addition, $f$ is homogenous of degree zero, i. e. $$f(\mathfrak{g}_{\gamma})\subseteq \mathfrak{h}_{\gamma},$$ holds for any $\gamma\in \Gamma$, then $f$ is said to be even.

Let $(\mathfrak{g},[\cdot,\cdot],\varepsilon,\alpha)$ and $(\mathfrak{g}',[\cdot,\cdot]',\varepsilon',\alpha')$ be two color Hom-Lie algebras. A homomorphism of degree zero $f:\mathfrak{g}\rightarrow \mathfrak{g}'$ is called a morphism of color Hom-Lie algebras if
\begin{enumerate}[label=\upshape{\arabic*)},left=7pt]
  \item  $[f(x),f(y)]'=f([x,y])$, for all $x,y \in \mathfrak{g},$
  \item  $f \circ \alpha=\alpha' \circ f.$
\end{enumerate}

In particular, if $\alpha$ is a morphism of color Lie algebras, then we call $(\mathfrak{g},[\cdot,\cdot],\varepsilon,\alpha)$, a multiplicative color Hom-Lie algebra.
\begin{exm}[\cite{AbdaouiAmmarMakhloufCohhomLiecolalg2015}]
As in case of Hom-associative and Hom-Lie algebras,
examples of multiplicative color Hom-Lie algebras can be constructed for example by the standard method of composing multiplication with algebra morphism.
Let $(\mathfrak{g},[\cdot,\cdot],\varepsilon)$ be a color Lie algebra and $\alpha$ be a color Lie algebra morphism. Then $(\mathfrak{g},[\cdot,\cdot]_{\alpha}:=\alpha\circ [\cdot,\cdot],\varepsilon,\alpha)$ is a multiplicative color Hom-Lie algebra.
\end{exm}
As for an associative algebra and a Lie algebra, a Hom-associative color algebra $(V,\mu,\alpha)$ gives a color Hom-Lie algebra by antisymmetrization. We denote this color Hom-Lie algebra by $(A,[\cdot,\cdot]_{A},\beta_{A})$, where $\beta_{A}=\alpha$ and $[x,y]_{A}=xy-yx$, for all $x,y\in A$.

\section{$(\sigma,\tau)$-differential graded commutative color algebra}
\label{sec:stdifgrcomcolalg}
\begin{defn}
Let $A$ be an associative algebra, and let $\sigma$, $\tau$ denote two algebra endomorphisms on $A$. A $(\sigma,\tau)$-derivation on $A$ is a linear map $D:A\rightarrow A$ such that
\begin{equation}\label{deriv}
  D(ab)=D(a)\tau (b)+\sigma (a)D(b),
\end{equation}
for all $a,b\in A$. A $\sigma$-derivation on $A$ is a $(\sigma,\id)$-derivation.
\end{defn}

In \cite{HLS}, Hom-Lie algebra or more general quasi Hom-Lie structures have been shown to arise in fundamental ways for $\sigma$-derivations on associative algebras. We define $(\sigma,\tau)$-differential graded commutative color algebras as follows.
\begin{defn}\label{FIHAA}
A $(\sigma,\tau)$-differential graded commutative color algebra $(\mathcal{A},\sigma,\tau,d_{\mathcal{A}})$ consists of a $\Gamma$-graded commutative algebra $\mathcal{A}$, two algebra endomorphisms $\sigma$ and $\tau$ of degree zero and an operator $d_{\mathcal{A}}$ of degree $p$ which satisfy the following conditions:
\begin{enumerate}[label=\upshape{(\roman*)},left=7pt]
  \item $d_{\mathcal{A}}^{2}=0$;
  \item $d_{\mathcal{A}}$ commutes with $\sigma$ and $\tau$;
  \item $d_{\mathcal{A}}(ab)= d_{\mathcal{A}}(a)\tau(b)+ \varepsilon(a,b)\sigma(a)d_{\mathcal{A}}(b)$, for homogeneous $a, b \in \mathcal{A}$.
\end{enumerate}
\end{defn}

For a $\Gamma$-graded vector space $\mathfrak{g}$, denote by $\bigwedge \mathfrak{g}*=\sum_{k}\bigwedge^{k}\mathfrak{g}^{*}$, $k\in \Gamma$ the exterior algebra of $\mathfrak{g}*$. For an endomorphism $\beta$, its dual map $\beta^{*}$ naturally extends to an algebra morphism,
$$(\beta^{*}\xi)(x_{1},\cdots,x_{k})= \xi(\beta(x_{1}),\cdots,\beta(x_{k})),$$
for all $\xi\in \bigwedge\mathfrak{g}^{*}$. Let
\begin{align*}
d\xi(x_{0},\dots,x_{p})= \sum_{i<j}(-1)^{i+j} \theta_{ij}(x) & \psi(\alpha(x_{0}),\dots,\alpha(x_{i-1}),\\
&[x_{i},x_{j}],\alpha(x_{i-1}),\dots,\widehat{x_{j}},\dots,\alpha(x_{p})).
\end{align*}

\begin{prop} \label{prop:dprops} The following properties hold:
\begin{enumerate}[label=\upshape{(\roman*)},left=7pt]
  \item \label{prop:dprops:i} $d^{2}=0,$
  \item \label{prop:dprops:ii} $\alpha^{*}\circ d=d \circ \alpha^{*},$
  \item \label{prop:dprops:iii} $d(\xi \wedge \eta)=d\xi \wedge \alpha^{*}\eta + \varepsilon(k,l) \alpha^{*}\xi\wedge d\eta$, for all $\xi\in \wedge^{k}\mathfrak{g}^{*}$ and $\eta\in \wedge^{l}\mathfrak{g}^{*}$.
\end{enumerate}
\end{prop}
\begin{proof}
\begin{enumerate}[label=\upshape{(\roman*)},left=7pt]
\item The proof for \ref{prop:dprops:i} can be found in \cite{AbdaouiAmmarMakhloufCohhomLiecolalg2015,AmmarEjbehiMakhlouf:homdeformation}.
\item  Let $\xi\in \wedge^{k}\mathfrak{g}^{*}$. We have
  \begin{align*}
    \alpha^{*} & \circ d \xi (x_{1},\cdots,x_{k+1})= d \xi (\alpha(x_{1}),\cdots,\alpha(x_{k+1})) \\
    = & \sum_{i<j} \varepsilon(i,j) \varepsilon(|x_{1}|+\cdots+|x_{i-1}|,|x_{i}|) \varepsilon(|x_{1}|+\cdots+|x_{j-1}|,|x_{j}|) \varepsilon(|x_{i}|,|x_{j}|) \\
     & \xi ([\alpha(x_{i}),\alpha(x_{j})], \alpha^{2}(x_{1}), \cdots, \widehat{x_{i}}, \cdots, \widehat{x_{j}}, \cdots, \alpha^{2}(x_{k+1})) \\
    = & \sum_{i<j} \varepsilon(i,j) \varepsilon(|x_{1}|+\cdots+|x_{i-1}|,|x_{i}|) \varepsilon(|x_{1}|+\cdots+|x_{j-1}|,|x_{j}|) \varepsilon(|x_{i}|,|x_{j}|) \\
     & \alpha^{*} \xi ([x_{i},x_{j}], \alpha(x_{1}), \cdots, \widehat{x_{i}}, \cdots, \widehat{x_{j}}, \cdots, \alpha(x_{k+1})) \\
    = & d(\alpha^{*} \xi)(x_{1},\cdots,x_{k+1}).
  \end{align*}
  \item  We use induction on $k$. If $k=1$, then $\xi\wedge \eta\in \wedge^{1+l}\mathfrak{g}^{*}$ and
      \begin{align*}
        d( \xi \wedge \eta)&(x_{1},\cdots,x_{l+2}) \\
        = & \sum_{i<j} \varepsilon(i,j) \varepsilon(|x_{1}|+\cdots+|x_{i-1}|,|x_{i}|) \varepsilon(|x_{1}|+\cdots+|x_{j-1}|,|x_{j}|) \varepsilon(|x_{i}|,|x_{j}|) \\
         & \hspace{1cm} \xi\wedge \eta ([x_{i},x_{j}], \alpha(x_{1}), \cdots, \widehat{x_{i}}, \cdots, \widehat{x_{j}}, \cdots, \alpha(x_{l+2})) \\
        = & \sum_{i<j} \varepsilon(i,j) \varepsilon(|x_{1}|+\cdots+|x_{i-1}|,|x_{i}|) \varepsilon(|x_{1}|+\cdots+|x_{j-1}|,|x_{j}|) \varepsilon(|x_{i}|,|x_{j}|) \\
         & \hspace{1cm} \{\xi([x_{i},x_{j}]) \eta( \alpha(x_{1}), \cdots, \widehat{x_{i}}, \cdots, \widehat{x_{j}}, \cdots, \alpha(x_{l+2})) \\
        &+ \sum_{p<i} (-1)^{p} \varepsilon(|x_{i}|+|x_{j}|+|x_{1}|+\cdots+|x_{p-1}|,|x_{p}|) \\
         & \hspace{1cm} \xi(\alpha(x_{p})) \eta([x_{i},x_{j}], \alpha(x_{1}),\cdots, \widehat{x_{p}}, \cdots, \widehat{x_{i}}, \cdots, \widehat{x_{j}}, \cdots, \alpha(x_{l+2})) \\
        &+ \sum_{i<p<j} (-1)^{p-1} \varepsilon(|x_{j}|+|x_{1}|+\cdots+|x_{p-1}|,|x_{p}|) \\
         & \hspace{1cm} \xi(\alpha(x_{p})) \eta([x_{i},x_{j}], \alpha(x_{1}),\cdots, \widehat{x_{p}}, \cdots, \widehat{x_{i}}, \cdots, \widehat{x_{j}}, \cdots, \alpha(x_{l+2})) \\
        &+ \sum_{j<p} (-1)^{p-2} \varepsilon(|x_{1}|+\cdots+|x_{p-1}|,|x_{p}|) \\
         & \hspace{1cm} \xi(\alpha(x_{p})) \eta([x_{i},x_{j}], \alpha(x_{1}),\cdots, \widehat{x_{p}}, \cdots, \widehat{x_{i}}, \cdots, \widehat{x_{j}}, \cdots, \alpha(x_{l+2}))\} \\
        = & d\xi \wedge \alpha^{*} \eta(x_{1},\cdots,x_{l+2}) -\alpha^{*}\xi\wedge d\eta(x_{1},\cdots,x_{l+2}).
      \end{align*}
Therefore, when $k=1$,
$$d(\xi\wedge \eta)=d\xi\wedge \alpha^{*}\eta + \varepsilon(1,l)\alpha^{*}\xi \wedge d\eta.$$
Suppose that for $k=n$,
$$d(\xi\wedge \eta)=d\xi\wedge \alpha^{*}\eta + \varepsilon(n,l)\alpha^{*}\xi \wedge d\eta.$$
Let $\omega \in \mathfrak{g}^{*}$. We have $\xi\wedge\omega\in  \wedge^{n+l} \mathfrak{g}^{*} $ and
\begin{align*}
  d(\xi\wedge\omega\wedge \eta )= & d\xi\wedge \alpha^{*} (\omega\wedge\eta)+\varepsilon(n,l)\alpha^{*}\xi\wedge d(\omega\wedge \eta)\\
  = & d\xi\wedge (\alpha^{*}\omega\wedge\alpha^{*}\eta)+\varepsilon(n,l) \alpha^{*}\xi\wedge (d\omega \wedge \alpha^{*}\eta +\varepsilon(1,l)\alpha^{*}\omega\wedge d\eta ) \\
  = & (d\xi\wedge\alpha^{*}\omega+ \varepsilon(n,l) \alpha^{*}\xi\wedge d\omega)\wedge\alpha^{*}\eta +\varepsilon(n+1,l) (\alpha^{*}\xi\wedge \alpha^{*}\omega) d\eta  \\
  = & d(\xi\wedge \omega)\wedge \alpha^{*}\eta +\varepsilon(n+1,l)\alpha^{*}\xi\wedge\omega)\wedge d\eta),
\end{align*}
which completes the proof.
\end{enumerate}
\end{proof}

Part \ref{prop:dprops:iii} of the above proposition says that  $(\wedge\mathfrak{g}^{*},\alpha^{*},d)$ is an $(\alpha^{*},\alpha^{*})$-differential graded commutative algebra. The converse of the above conclusion is also true. Thus, we have the following theorem, which generalizes the classical result about the relation between Lie algebra structures and DGCA structures.

\begin{thm}
  Triple $(\mathfrak{g},[\cdot,\cdot],\alpha)$ is a color Hom-Lie algebra if and only if $(\wedge\mathfrak{g}^{*},\alpha^{*},d)$ is an $(\alpha^{*},\alpha^{*})$-differential graded commutative color algebra.
\end{thm}
\begin{proof}
Suppose that $(\wedge\mathfrak{g}^{*},\alpha^{*},d)$ is an $(\alpha^{*},\alpha^{*})$-differential graded commutative color algebra. One can define a skewsupersymmetric bracket $[\cdot,\cdot]:\wedge^{2}\mathfrak{g}\rightarrow \mathfrak{g}$ by
$$\langle\eta,[x_{1},x_{2}]\rangle= -d\eta(x_{1},x_{2}),$$
for all $\eta \in \mathfrak{g}^{*}$, $x_{1},x_{2}\in \mathfrak{g}$.
We have
\begin{eqnarray*} d(\alpha^{*}\eta)(x_{1},x_{2}) &=& -\langle\alpha^{*} \eta ,[x_{1},x_{2}]\rangle=-\langle\eta,\alpha[x_{1},x_{2}]\rangle, \\ \alpha^{*}d\eta(x_{1},x_{2}) &=& -\langle \eta,[\alpha(x_{1}),\alpha(x_{2})]\rangle.
\end{eqnarray*}
Moreover,
$\alpha([x_{1},x_{2}])=[\alpha(x_{1}),\alpha(x_{2})],$
which implies that $\alpha$ is an algebra endomorphism.
On the other hand, for $\xi\in \mathfrak{g}^{*}$, $x_{1},x_{2},x_{3}\in \mathfrak{g}$ we have
\begin{align*}
  &0=d (d\xi)(x_{1},x_{2},x_{3})= \varepsilon(1,2)d\xi([x_{1},x_{2}],\alpha(x_{3})) \\
  & \hspace{2cm} + \varepsilon(1,3)\varepsilon(|x_{1}|+|x_{2}|,|x_{3}|) \varepsilon(|x_{1}|,|x_{3}|) d\xi([x_{1},x_{3}],\alpha(x_{2}))\\
  & \hspace{3cm} + \varepsilon(2,3)\varepsilon(|x_{1}|,|x_{2}|) \varepsilon(|x_{1}|+|x_{2}|,|x_{3}|) \varepsilon(|x_{2}|,|x_{3}|) d\xi([x_{2},x_{3}],\alpha(x_{1}))\\
  &= \varepsilon(1,2)d\xi([x_{1},x_{2}],\alpha(x_{3}))+ \varepsilon(|x_{2}|,|x_{3}|)d\xi([x_{1},x_{3}],\alpha(x_{2}))\\
  & \hspace{4cm} +  \varepsilon(|x_{2}|+|x_{3}|,|x_{1}|) d\xi([x_{2},x_{3}],\alpha(x_{1})) \\
  &= \xi(([x_{1},x_{2}],\alpha(x_{3})) +\varepsilon(|x_{2}|,|x_{3}|)([x_{1},x_{3}],\alpha(x_{2})) \\
  & \hspace{3cm} +  \varepsilon(|x_{2}|+|x_{3}|,|x_{1}|) \xi([x_{2},x_{3}],\alpha(x_{1})),
\end{align*}
which implies that
\begin{equation*}
([x_{1},x_{2}],\alpha(x_{3}))+
\varepsilon(|x_{2}|,|x_{3}|)  ([x_{1},x_{3}],\alpha(x_{2}))+ \varepsilon(|x_{2}|+|x_{3}|,|x_{1}|)([x_{2},x_{3}],\alpha(x_{1}))=0,
\end{equation*}
This completes the proof.
\end{proof}

\section{Representations of color Hom-Lie algebras}
\label{sec:ReprescolorHomLiealg}
We are going to generalize some results from \cite{AmmarMabroukMakhloufCohomnaryHNLalg2011,Sheng:homrep}.
\begin{defn}
Let $(\mathfrak{g},[\cdot,\cdot],\varepsilon,\alpha)$ be a color Hom-Lie algebra. A representation of $\mathfrak{g}$ is a triplet $(M,\rho,\beta)$, where $M$ is a $\Gamma$-graded vector space, $\beta\in End(M)_{0}$ and $\rho:\mathfrak{g}\rightarrow End(M)$ is an even linear map satisfying
\begin{equation}
  \rho([x,y])\circ \beta=\rho(\alpha(x))\circ \rho(y)- \varepsilon(x,y)\rho(\alpha(y))\circ \rho(x),
\end{equation}
for all homogeneous elements $x,y \in \mathfrak{g}$.
\end{defn}
Let $\mathfrak{g}$ be a $\Gamma$-graded vector space and let $\beta\in \mathfrak{gl}(\mathfrak{g})_{\bar{0}}$.
For any homogenous elements $x,y\in \mathfrak{gl}(\mathfrak{g})$, define
$[\cdot,\cdot]_{\beta}:\mathfrak{gl}(\mathfrak{g}) \times \mathfrak{gl}(\mathfrak{g}) \rightarrow \mathfrak{gl}(\mathfrak{g}),$
by
$$[x,y]_{\beta}=\beta x\beta^{-1}y\beta^{-1}-\varepsilon(x,y)\beta y\beta^{-1} x \beta^{-1},$$
Recall from \cite{AbdaouiAmmarMakhloufCohhomLiecolalg2015}, the adjoint action on $\mathfrak{gl}(\mathfrak{g})$:
$$Ad_{\beta}(x)=\beta x\beta^{-1},$$
for an element $x$ which satisfies $\alpha(x)=x$.

It is shown in \cite{AbdaouiAmmarMakhloufCohhomLiecolalg2015,AmmarEjbehiMakhlouf:homdeformation} that $(\mathfrak{g},ad_{k},\alpha)$ is a representation of $\mathfrak{g}$ which is called the adjoint representation of the color Hom-Lie algebra $\mathfrak{g}$.

\begin{prop}
  Let $\mathfrak{g}$ and $[\cdot,\cdot]_{\beta}$ be as described above. Then $(\mathfrak{gl}(\mathfrak{g}),[\cdot,\cdot]_{\beta}, Ad_{\beta} )$ is a regular color hom Lie algebra.
\end{prop}
\begin{proof}
  One can easily see that $Ad_{\beta}$ is invertible, since
  $Ad_{\beta} \circ Ad_{\beta^{-1}}=\id$. Moreover,
  \begin{align*}
    [Ad_{\beta}(x) & Ad_{\beta}(y)]_{\beta}=[\beta x \beta^{-1}, \beta y \beta^{-1}]_{\beta} \\
    = & \beta^{2}x \beta^{-1}y\beta^{-1}\beta^{-1} -\varepsilon(x,y) \beta^{2} y \beta^{-1} x \beta^{-1}\beta^{-1}= Ad_{\beta}([x,y]_{\beta}),
  \end{align*}
  for all $x,y \in \mathfrak{gl}(\mathfrak{g})$. Furthermore,
  \begin{align*}
  \displaystyle{\sum_{cyclic\{x,y,z\}}} & \varepsilon(x,z) [[x,y]_{\beta},Ad_{\beta}(z)]_{\beta}  \\
    = &  \displaystyle{\sum_{cyclic\{x,y,z\}}} (\varepsilon(x,z)([\beta x \beta^{-1},\beta y \beta^{-1}z\beta^{-1}]_{\beta})- [\beta x \beta^{-1}, \varepsilon (z,y)\beta z \beta^{-1}y \beta^{-1}]_{\beta})  \\
    = &  \displaystyle{\sum_{cyclic\{x,y,z\}}} (\varepsilon(x,z) \beta^{2}x \beta^{-1}y \beta^{-1}z \beta^{-1}\beta^{-1}- \varepsilon(z,y+x)\beta^{2}x \beta^{-1}z\beta^{-1}y \beta^{-1}\beta^{-1} \\
     & -\varepsilon(x,y)\beta^{2}y \beta^{-1}z \beta^{-1}x \beta^{-1}\beta^{-1}
      +\varepsilon(y,x+y) \beta^{2}z\beta^{-1} y \beta^{-1} x \beta^{-1}\beta^{-1}) =  0.
  \end{align*}
This completes the proof.
\end{proof}

\begin{thm}
  Let $(\mathfrak{g},[\cdot,\cdot],\alpha)$ be a color Hom-Lie algebra, $V$ a $\Gamma$-graded vector space and $\beta \in \mathfrak{gl}(V)_{\bar{0}}$. Then $\rho:\mathfrak{g}\rightarrow \mathfrak{gl}(V) $ is a representation of $(\mathfrak{g},[\cdot,\cdot],\alpha)$ on $V$ with respect to $\beta$ if and only if $\rho:(\mathfrak{g},[\cdot,\cdot],\alpha) \rightarrow (\mathfrak{gl}(V),[\cdot,\cdot]_{\beta},Ad_{\beta}) $ is a morphism of color Hom-Lie algebras.
\end{thm}
\begin{proof}
Let $\rho:\mathfrak{g}\rightarrow \mathfrak{gl}(V)$ be a representation of $(\mathfrak{g},[\cdot,\cdot],\alpha)$ on $V$ with respect to $\beta$. One can see that
\begin{equation}\label{haft}
  \rho(\alpha(x))\circ\beta=\beta\circ\rho(x),
\end{equation}
\begin{equation}\label{hasht}
  \rho([x,y])\circ\beta=\rho(\alpha(x))\rho(y) -\varepsilon(x,y)\rho(\alpha(y))\rho(x).
\end{equation}
Using \eqref{haft}, we get that $\rho\circ \alpha=Ad_{\beta}\circ \rho$. Moreover, due to \eqref{haft} and \eqref{hasht},
\begin{align*}
  \rho([x,y])= & \rho(\alpha(x))\beta\circ \beta^{-1}\rho(y)\beta^{-1} -\varepsilon(x,y)\rho(\alpha(y))\beta\circ\beta^{-1}\rho(x)\beta^{-1} \\
  = & \beta \rho(x)\beta^{-1}\rho(y)\beta^{-1} -\varepsilon(x,y)\beta\rho(y)\beta^{-1}\rho(x)\beta^{-1} \\
  = & [\rho(x),\rho(y)]_{\beta}.
\end{align*}
Hence, $\rho$ is a morphism of color Hom-Lie algebras. The converse is shown easily in a similar way.
\end{proof}

\begin{cor}
  Let $(\mathfrak{g},[\cdot,\cdot],\alpha)$ be a regular color Hom-Lie algebra. Then the adjoint representation $ad:\mathfrak{g}\rightarrow \mathfrak{gl}(\mathfrak{g})$ is a morphism from $(\mathfrak{g},[\cdot,\cdot],\alpha)$ to $(\mathfrak{gl}(g),[\cdot,\cdot]_{\alpha},Ad_{\alpha})$.
\end{cor}
%%%%%%%%%%%%%%%%%%%%%%%%%%%%%%%%%%%%%%%%%%%%%%%%%%%%%%%%%%%%%%%%%%%%%%%%%%%%%%%%%%%%%%%%%%%%%%%%%%%%%%%%%%%%
Let $\rho:\mathfrak{g}\rightarrow\mathfrak{gl}(V)$ be a representation of $(\mathfrak{g},[\cdot,\cdot],\alpha)$ on $V$ with respect to $\beta\in \mathfrak{gl}(V)_{\bar{0}}$. Denote by $C^{k}(\mathfrak{g};V)$, the set of all $k$-cochains on $\mathfrak{g}$, i.e. all $k$-linear homogeneous maps
$$\varphi:\bigwedge^{n}(\mathfrak{g})\rightarrow V,$$
satisfying
$\varphi(x_{1},\ldots,x_{i},x_{i+1},\ldots,x_{k}) = -\varepsilon(x_{i},x_{i+1}) \varphi(x_{1},\ldots, x_{i+1},x_{i},\ldots,x_{k}).$

Define
$$C^{n}(\mathfrak{g};V)_{\theta}=\{\varphi\in C^{n}(\mathfrak{g};V): |\varphi(x_{1},\ldots,x_{n})|= |x_{1}|+\ldots+|x_{n}|+\theta\},$$
for all $\theta\in\Gamma$.

If $\beta\in \mathfrak{gl}(V)_{\bar{0}}$, we define $\bar{\beta}$ from $C^{k}(\mathfrak{g};V)$ to itself using the $k$-cochains
\begin{equation}\label{noh}
\bar{\beta}(\varphi)(x_{1},\ldots,x_{k})=\beta\circ \varphi(x_{1},\ldots,x_{k}),
\end{equation}
for all $\varphi\in C^{n}(\mathfrak{g};V)$.
Moreover, using $\alpha$, one can define
$$\bar{\alpha}:C^{k}(\mathfrak{g};V)\rightarrow C^{k}(\mathfrak{g};V)$$
\begin{equation}\label{dah}
\bar{\alpha}(\varphi)(x_{1},\ldots,x_{k}) =\varphi(\alpha(x_{1}),\ldots,\alpha(x_{k})),
\end{equation}
for all $\varphi\in C^{n}(\mathfrak{g};V)$.

\begin{defn}
A $k$-Hom-cochain on $\mathfrak{g}$ with values in $V$ is a $k$-cochain $\varphi\in C^{k}(\mathfrak{g};V)$ such that
$\bar{\alpha}(\varphi)=\bar{\beta}(\varphi).$
\end{defn}
The set of all $k$-Hom-cochains on $\mathfrak{g}$ with values in $V$ is denoted by $C^{k}_{\alpha,\beta}(\mathfrak{g};V)$.
The action $\bullet:C^{l}_{\alpha}(\mathfrak{g};V)\times C^{k}_{\alpha,\beta}(\mathfrak{g};V) \rightarrow C^{k+l}_{\alpha,\beta}(\mathfrak{g};V)$ is defined as follows.
For $l=1$,
$$\eta\bullet\varphi(x_{1},\ldots,x_{k+1})=\sum_{i}sgn(i) \eta(x_{i_{1}})\varphi(x_{i_{l+1}},\ldots,x_{i_{k+1}}).$$
For $l\geqslant2$,
$$\eta\bullet\varphi(x_{1},\ldots,x_{k+1})=  \sum_{i}sgn(i) \kappa_{j}(x)\eta(x_{i_{1}},\ldots ,x_{i_{l}} ) \varphi(x_{i_{l+1}},\ldots,x_{i_{k+1}}),$$
for all $\eta \in C^{l}_{\alpha}(\mathfrak{g};V)$, $\varphi\in C^{k}_{\alpha,\beta}(\mathfrak{g};V)$, where
$$\lambda_{j}(x)= \varepsilon((\sum_{j=1}^{l}|x_{1}|+\ldots+|x_{k+1}|),|x_{j}|) (-1)^{\sum_{1\leq p\neq q\leq l}|x_{p}||x_{q}|},$$ and the summation is taken over $(l,k)$-shuffles.

The linear map $d_{s}:C^{k}_{\alpha,\beta}(\mathfrak{g};V) \rightarrow C^{k+1}(\mathfrak{g};V)$ is defined as follows
\begin{align*}
d_{s}\varphi(x_{1},\ldots,x_{k+1}) &= \sum_{i=1}^{k+1} \kappa_{j}(x)\rho(\alpha^{s+k}(x_{i})) \varphi(x_{1},\ldots,\hat{x_{i}},\ldots,x_{k+1}) \\
 & + \sum_{i<j} (-1)^{i}\kappa_{ji}(x) \varphi([x_{i},x_{j}],\alpha(x_{1}),\ldots,\hat{x_{i}},\ldots, \hat{x_{j}},\ldots,\alpha(x_{k+1})), \\
\kappa_{j}(x)  & =  \varepsilon(x_{1}+\dots+x_{j-1},x_{j}),\quad
\kappa_{ji}(x) =  \varepsilon(x_{j+1}+\dots+x_{i-1},x_{i}).
\end{align*}
It is shown in \cite{AbdaouiAmmarMakhloufCohhomLiecolalg2015,AmmarEjbehiMakhlouf:homdeformation,ArmakanSilvFarh:exthomLiecoloralg} that $d_{s}$ is a well-defined coboundry operator. Moreover, it can be easily checked that
$$\beta\circ d_{s}=d_{s+1}\circ \bar{\beta}.$$
It is also shown that $d_{s}$ is an $\alpha^{l}$ derivation in the sense that
$$d_{s}(\eta\bullet\varphi)=d\eta\bullet\bar{\alpha}(\varphi)+ \varepsilon(s,l)\eta\bullet d^{s+l}\varphi.$$

%%%%%%%%%%%%%%%%%%%%%%%%%%%%%%%%%%%%%%%%%%%%%%%%%%%%%%%%%%%%%%%%%%%%%%
Let $(\mathfrak{g},[\cdot,\cdot],\alpha)$ be a color Hom-Lie algebra. Denote by $C^{k}_{\alpha}(\mathfrak{g})$ the set of all $\xi\in \bigwedge^{k}\mathfrak{g}^{*}$ for which $\alpha^{*}\xi=\xi$.
Then the complex $(\bigoplus_{k}C^{k}_{\alpha}(\mathfrak{g}),d)$ is a subcomplex of $(\bigwedge \mathfrak{g}^{*},d)$, where $(\bigwedge \mathfrak{g}^{*},\wedge)$ is the exterior Hom-algebra. This complex is considered to be the cohomological complex of $\mathfrak{g}$ in\cite{AbdaouiAmmarMakhloufCohhomLiecolalg2015,ArmakanSilvFarh:exthomLiecoloralg}.

Let $\rho:\mathfrak{g}\rightarrow \mathfrak{gl}(V)$ be a representation of the color Hom-Lie algebra $(\mathfrak{g},[\cdot,\cdot],\alpha)$ on the $\Gamma$-graded vector space $V$ with respect to $\beta\in \mathfrak{gl}(V)_{\bar{0}}$. Denote by $C^{k}(\mathfrak{g},V)$, The set of $k$-cochains on $\mathfrak{g}$ with values in $V$. Therefore, $C^{k}(\mathfrak{g},V)$ is spanned by all $k$-linear homogenous maps $\varphi:\Lambda^{k}\mathfrak{g}\rightarrow V$ for which one has
$$\varphi(x_{1},\ldots,x_{i},x_{i+1},\ldots,x_{k}) =\varepsilon(x_{i+1},x_{i}) \varphi(x_{1},\ldots,x_{i+1},x_{i},\ldots,x_{k}). $$

%%%%%%%%%%%%%%%%%%%%%%%%%%%%%%%%%%%%%%%%%%%%%%%%%%%%%%%%%%%%%%%%%%%%%%
\section{Color Omni-Hom-Lie algebras}
\label{sec:colorOmniHomLiealg}

\begin{defn}
  Let $\mathfrak{g}$ be a $\Gamma$-graded vector space and $\beta\in \mathfrak{gl}(\mathfrak{g})_{\bar{0}}$. An color omni-Hom-Lie algebra is a quadruple $(\mathfrak{gl}(\mathfrak{g})\otimes \mathfrak{g}, \delta_{\beta}, \{\cdot,\cdot\}_{\beta},\langle\cdot,\cdot\rangle)$, where
  $$\delta_{\beta}: \mathfrak{gl}(\mathfrak{g})\otimes \mathfrak{g} \rightarrow \mathfrak{gl}(\mathfrak{g})\otimes \mathfrak{g},$$
  is an even linear map satisfying
  $$\delta_{\beta}(A+x)=Ad_{\beta}(A)+\beta(x),$$
  for all $A+x\in \mathfrak{gl}(\mathfrak{g})\otimes \mathfrak{g}$,
  $$\{\cdot,\cdot\}_{\beta}: \mathfrak{gl}(\mathfrak{g})\otimes \mathfrak{g} \times \mathfrak{gl}(\mathfrak{g})\otimes \mathfrak{g}\rightarrow \mathfrak{gl}(\mathfrak{g})\otimes \mathfrak{g},$$
  is a bilinear map satisfying
  $$\{A+x,B+y\}_{\beta}=[x,y]_{\beta}+A(y),$$
  for all $A+x, B+y\in \mathfrak{gl}(\mathfrak{g})\otimes \mathfrak{g}$,
  and
  $$\langle\cdot,\cdot\rangle:\mathfrak{gl}(\mathfrak{g})\otimes \mathfrak{g} \times \mathfrak{gl}(\mathfrak{g})\otimes \mathfrak{g}\rightarrow \mathfrak{g},$$
  is a supersymmetric bilinear $V$-valued pairing given by
  $$\langle A+x,B+y\rangle = \frac{1}{2}(A(y) - \varepsilon(x,y)B(x)).$$
\end{defn}

\begin{defn}
A Hom-associative color algebra is a triple $(V,\mu,\alpha)$ consisting of a color space $V$, an even bilinear map $\mu:V\times V\rightarrow V$ and an even homomorphism $\alpha:V\rightarrow V$ satisfying
$$\mu(\alpha(x),\mu(y,z))=\mu(\mu(x,y),\alpha(z)),$$
for all $x,y,z \in V$.
\end{defn}

\begin{prop} \label{prop:deltabetacolorHomLeibniz}
The following properties hold:
  \begin{enumerate}[label=\upshape{(\roman*)},left=7pt]
    \item \label{prop:deltabetacolorHomLeibniz:i} $\delta_{\beta}$ is an algebra automorphism.
    \item \label{prop:deltabetacolorHomLeibniz:ii} $(\mathfrak{gl}(V)\oplus V,\{\cdot,\cdot\}_{\beta},\delta_{\beta})$ is a color Hom-Leibniz algebra. Moreover, we have
        $$\beta\left<A+u,B+v\right>= \left<\delta_{\beta}(A+u),\delta_{\beta}(B+v)\right>.$$
  \end{enumerate}
\end{prop}
\begin{proof}
  Since $Ad_{\beta}$ is an algebra automorphism,
  \begin{align*}
    \delta_{\beta}(\{A+u,B+v\}_{\beta})=&\delta_{\beta}([A,B]_{\beta} +A(v))= Ad_{\beta}([A,B]_{\beta})+\beta A(v)\\
    = & [Ad_{\beta}(A),_{\beta}(B)]_{\beta}+\beta A(v).
  \end{align*}
  On the other hand,
  \begin{multline*}
    \{\delta_{\beta}(A+u), \delta_{\beta}(B+v)\}_{\beta}=  \{Ad_{\beta}(A)+\beta(u),Ad_{\beta}(B)+\beta(v)\}_{\beta} \\
    =[Ad_{\beta}(A),Ad_{\beta}(B)]_{\beta} +Ad_{\beta}(A)\beta(v)
    =  [Ad_{\beta}(A),Ad_{\beta}(B)]_{\beta}+\beta A(v).
  \end{multline*}
  Therefore, $\delta_{\beta}$ is an algebra automorphism.
Moreover,
\begin{align*}
& \{\delta_{\beta}(A+u),\{B+v,C+w\}_{\beta}\}_{\beta}= \{Ad_{\beta}(A)+\beta(u),[B,C]_{\beta}+B(w)\}_{\beta} \\
& \hspace{1cm}  =  [Ad_{\beta}(A),[B,C]_{\beta}]_{\beta}+Ad_{\beta}(A)B(w)
=  [Ad_{\beta}(A),[B,C]_{\beta}]_{\beta}+\beta A\beta ^{-1}B(w), \\
& \{\{A+u,B+v\}_{\beta},\delta_{\beta}(C+w)\}_{\beta}=  \{[A,B]_{\beta}+A(v),Ad_{\beta}(C)+\beta(w)\}_{\beta} \\
& \hspace{3cm}  = [[A,B]_{\beta},Ad_{\beta}(C)]_{\beta}+  [A,B]_{\beta}\beta(w) \\
& \hspace{3cm}  = [[A,B]_{\beta},Ad_{\beta}(C)]_{\beta}+  \beta A\beta^{-1}B(w) -\varepsilon(A,B)\beta B\beta^{-1}A(w), \\
&\{\delta_{\beta}(B+v),\{A+u,C+w\}_{\beta}\}_{\beta}=  \{Ad_{\beta}(B)+\beta(v),[A,C]_{\beta}+A(w)\}_{\beta} \\
& \hspace{5cm} =  [Ad_{\beta}(B),[A,C]_{\beta}]_{\beta}+Ad_{\beta}(B)A(w) \\
& \hspace{5cm}  =  [Ad_{\beta}(B),[A,C]_{\beta}]_{\beta}+\beta B\beta ^{-1}A(w),
\end{align*}
Since $(\mathfrak{gl}(V),[\cdot,\cdot]_{\beta},Ad_{\beta})$ is a Hom-Leibniz color algebra, \ref{prop:deltabetacolorHomLeibniz:ii} is proved.
Furthermore,
\begin{multline*}
  \langle\delta_{\beta}(A+u),\delta_{\beta}(B+v)\rangle =  \langle Ad_{\beta}(A)+\beta(u),Ad_{\beta}(B)+ \beta(v)\rangle \\
  =  \frac{1}{2}(Ad_{\beta}(A)\beta(v) -\varepsilon(A,B)Ad_{\beta}(B)\beta(u))
  =  \frac{1}{2}\beta(A(v)-\varepsilon(A,B)B(u)) \\
  =  \beta\langle A+u,B+v\rangle.
\end{multline*}
This completes the proof.
\end{proof}

\section{Acknowledgement}
The authors are grateful to Iran National Science Foundation (INSF) for support of the research in this article.

% ----------------------------------------------------------------
%\bibliographystyle{amsplain}

\end{document}